\documentclass[11pt,leqno]{amsart}

\usepackage{amsthm}
\usepackage{amsmath}
\usepackage{amssymb}
\usepackage[all,cmtip]{xy}
\usepackage[T1]{fontenc}
\usepackage{mathrsfs}
\usepackage{enumerate}
\theoremstyle{plain}
\newtheorem{thm}{Theorem}

\newtheorem{prop}{Proposition}
\newtheorem{cor}{Corollary}
\theoremstyle{definition}

\newtheorem{rem}{Remark}

\title{Commutativity and Ideals in Category Crossed Products}
\date{}

\author{Johan \"Oinert and Patrik Lundstr\"om}

\address{Johan \"{O}inert,
Centre for Mathematical Sciences,
Lund University,
P.O. Box 118,
SE-22100 Lund,
Sweden}
\email{Johan.Oinert@math.lth.se}

\address{Patrik Lundstr\"{o}m,
University West,
Department of Engineering Science,
SE-46186 Trollh\"{a}ttan,
Sweden}
\email{Patrik.Lundstrom@hv.se}

\subjclass[2000]{16W50, 16S99}

\begin{document}

\maketitle

%\tableofcontents

\begin{abstract}
In order to simultaneously generalize matrix rings and group graded
crossed products, we introduce
category crossed products. For such algebras we describe
the center and the commutant of the coefficient ring.
We also investigate the connection between on the one hand
maximal commutativity of the coefficient ring
and on the other hand
nonemptyness of intersections of the coefficient ring
by nonzero twosided ideals.
\end{abstract}

\section{Introduction}

Let $R$ be a ring.
By this we always mean that $R$ is an additive group
equipped with a multiplication which is associative and unital.
The identity element of $R$ is denoted $1_R$ and
the set of ring endomorphisms of $R$ is denoted $End(R)$.
We always assume that ring homomorphisms respect the multiplicative identities.
The center of $R$ is denoted $Z(R)$ and by the commutant
of a subset of $R$ we mean the collection of elements
in $R$ that commute with all the elements in the subset.

Suppose that $R_1$ is a subring of $R$ i.e. that there is
an injective ring homomorphism $R_1 \rightarrow R$.
Recall that if $R_1$ is commutative, then it
is called a maximal commutative subring of $R$
if it coincides with its commutant in $R$.
A lot of work has been devoted to investigating the
connection between on the one hand maximal
commutativity of $R_1$ in $R$ and on the other hand
nonemptyness of intersections of $R_1$ with
nonzero twosided ideals of $R$
(see 
\cite{coh},
\cite{fis},
%\cite{for78},
%\cite{lor},
\cite{irv79a},
\cite{irv79b},
\cite{lau06},
\cite{lor79},
\cite{lor80}
and \cite{mon78}).
%, \cite{neij07}
%and \cite{pas77}).
Recently (see \cite{oin06}, \cite{oin07}, \cite{oin08},
\cite{oin09} and \cite{oin10})
such a connection was established for the commutant $R_1$
of the coefficient ring of crossed products $R$
(see Theorem 1 below).
Recall that crossed products are defined by
first specifying a crossed system i.e. a quadruple
$\{ A,G,\sigma,\alpha \}$ where $A$ is a ring, $G$ is a group (written multiplicatively and with identity element $e$) and
$\sigma : G \rightarrow End(A)$ and $\alpha : G \times G \rightarrow
A$ are maps satisfying the following four conditions:
\begin{equation}\label{id}
\sigma_e = {\rm id}_A
\end{equation}
\begin{equation}\label{identity}
\alpha(s,e) = \alpha(e,s) = 1_A
\end{equation}
\begin{equation}\label{associative}
\alpha(s,t) \alpha(st,r) = \sigma_s(\alpha(t,r)) \alpha(s,tr)
\end{equation}
\begin{equation}\label{algebra}
\sigma_s(\sigma_t(a)) \alpha(s,t) = \alpha(s,t) \sigma_{st}(a)
\end{equation}
for all $s,t,r \in G$ and all $a \in A$.
The crossed product, denoted
$A \rtimes_{\alpha}^{\sigma} G$, associated to this quadruple,
is the collection of formal sums
$\sum_{s \in G} a_s u_s$, where $a_s \in A$, $s \in G$,
are chosen so that all but finitely many of them are nonzero.
By abuse of notation we write $u_s$ instead of $1u_s$
for all $s \in G$.
The addition on $A \rtimes_{\alpha}^{\sigma} G$ is defined pointwise
\begin{equation}\label{addition}
\sum_{s \in G} a_s u_s + \sum_{s \in G} b_s u_s =
\sum_{s \in G} (a_s + b_s)u_s
\end{equation}
and the multiplication on $A \rtimes_{\alpha}^{\sigma} G$ is defined
by the bilinear extension of the relation
\begin{equation}\label{multiplication}
(a_s u_s)(b_t u_t) = a_s \sigma_s(b_t) \alpha(s,t) u_{st}
\end{equation}
for all $s,t \in G$ and all $a_s,b_t \in A$.
By (\ref{id}) and (\ref{identity}) $u_e$ is a multiplicative identity of
$A \rtimes_{\alpha}^{\sigma} G$ and by (\ref{associative})
the multiplication on $A \rtimes_{\alpha}^{\sigma} G$
is associative. There is also an $A$-bimodule structure on $A \rtimes_{\alpha}^{\sigma} G$
defined by the linear extension of the relations
$a(b u_s) = (ab) u_s$ and
$(a u_s)b = (a \sigma_s(b)) u_s$ for
all $a,b \in A$ and all $s,t \in G$,
which, by (\ref{algebra}), makes
$A \rtimes_{\alpha}^{\sigma} G$ an $A$-algebra.
In the article \cite{oin06}, \"{O}inert and Silvestrov
show the following result.

\begin{thm}
If $A \rtimes_{\alpha}^{\sigma} G$ is a crossed product
with $A$ commutative, all $\sigma_s$, $s \in G$,
are ring automorphisms and all $\alpha(s,s^{-1})$, $s \in G$,
are units in $A$, then every
intersection of a nonzero twosided ideal of $A \rtimes_{\alpha}^{\sigma} G$
with the commutant of $A$ in $A \rtimes_{\alpha}^{\sigma} G$ is
nonzero.
\end{thm}

In loc. cit. \"{O}inert and Silvestrov
determine the center of crossed products and
in particular when crossed products are commutative;
they also give a description of the commutant of $A$
in $A \rtimes_{\alpha}^{\sigma} G$.
Theorem 1 has been generalized somewhat by relaxing
the conditions on $\sigma$ and $\alpha$
(see
%\cite{oin07},
\cite{oin08} and \cite{oin09})
and by considering general strongly group graded rings
(see \cite{oin10}).
For more details concerning group graded rings in general
and crossed product algebras in particular, see e.g. \cite{CaenOyst}, \cite{Karp} and \cite{nas}.

Many natural examples of rings, such as rings of
matrices, crossed product algebras defined by separable extensions
and category rings,
are not in any natural way graded by groups, but instead by
categories (see \cite{lu05}, \cite{lu06}, \cite{lu07} and
Remark \ref{firstremark}).
The purpose of this article is to
define a category graded generalization of
crossed products
and to analyze commutativity questions similar
to the ones discussed above for such algebras.
In particular, we wish to generalize Theorem 1
from groups to groupoids (see Theorem \ref{intersection} in Section 4).
To be more precise, suppose that $G$ is a category.
The family of objects of $G$ is denoted $ob(G)$;
we will often identify an object in $G$ with
its associated identity morphism.
The family of morphisms in $G$ is denoted $ob(G)$;
by abuse of notation, we will often write $s \in G$
when we mean $s \in mor(G)$.
The domain and codomain of a morphism $s$ in $G$ is denoted
$d(s)$ and $c(s)$ respectively.
We let $G^{(2)}$ denote the collection of composable
pairs of morphisms in $G$ i.e. all $(s,t)$ in
$mor(G) \times mor(G)$ satisfying $d(s)=c(t)$.
Analogously, we let $G^{(3)}$ denote the collection
of all composable triples of morphisms in $G$ i.e. all
$(s,t,r)$ in $mor(G) \times mor(G) \times mor(G)$
satisfying $(s,t) \in G^{(2)}$ and $(t,r) \in G^{(2)}$.
Throughout the article $G$ is assumed to be small i.e.
with the property that $mor(G)$ is a set.
By a crossed system we mean a
quadruple $\{ A,G,\sigma,\alpha \}$ where
$A$ is the direct sum of rings $A_e$, $e \in ob(G)$,
$\sigma_s : A_{d(s)} \rightarrow A_{c(s)}$, $s \in G$,
are ring homomorphisms and
$\alpha$ is a map from $G^{(2)}$ to the
disjoint union of the sets $A_e$, $e \in ob(G)$,
with $\alpha(s,t) \in A_{c(s)}$,
$(s,t) \in G^{(2)}$, satisfying the following five conditions:
\begin{equation}\label{idd}
\sigma_e = {\rm id}_{A_e}
\end{equation}
\begin{equation}\label{identityr}
\alpha(s,d(s)) = 1_{A_{c(s)}}
\end{equation}
\begin{equation}\label{identityl}
\alpha(c(t),t) = 1_{A_{c(t)}}
\end{equation}
\begin{equation}\label{associativee}
\alpha(s,t) \alpha(st,r) = \sigma_s(\alpha(t,r)) \alpha(s,tr)
\end{equation}
\begin{equation}\label{algebraa}
\sigma_s(\sigma_t(a)) \alpha(s,t) = \alpha(s,t) \sigma_{st}(a)
\end{equation}
for all $e \in ob(G)$, all $(s,t,r) \in G^{(3)}$ and all $a \in A_{d(t)}$.
Let $A \rtimes_{\alpha}^{\sigma} G$ denote the collection of formal sums
$\sum_{s \in G} a_s u_s$, where $a_s \in A_{c(s)}$, $s \in G$,
are chosen so that all but finitely many of them are nonzero.
Define addition on $A \rtimes_{\alpha}^{\sigma} G$
by (\ref{addition}) and define multiplication on $A \rtimes_{\alpha}^{\sigma} G$
by (\ref{multiplication})
if $(s,t) \in G^{(2)}$ and $(a_s u_s)(b_t u_t) = 0$ otherwise
where $a_s \in A_{c(s)}$ and $b_t \in A_{c(t)}$.
By (\ref{idd}), (\ref{identityr}) and (\ref{identityl})
it follows that $A \rtimes_{\alpha}^{\sigma} G$
has a multiplicative identity if and only if $ob(G)$
is finite;
in that case the multiplicative identity is
$\sum_{e \in ob(G)} u_e$.
By (\ref{associativee}) the multiplication on
$A \rtimes_{\alpha}^{\sigma} G$ is associative.
Define a left $A$-module structure on
$A \rtimes_{\alpha}^{\sigma} G$ by the bilinear extension
of the rule
$a_e (b_s u_s) = (a_s b_s) u_s$
if $e = c(s)$ and
$a_e (b_s u_s) = 0$ otherwise
for all $a_e \in A_e$, $b_s \in A_{c(s)}$,
$e \in ob(G)$, $s \in G$.
Analogously, define a right $A$-module structure on
$A \rtimes_{\alpha}^{\sigma} G$ by the bilinear extension
of the rule
$(b_s u_s) c_f = (b_s \sigma_s(c_f))u_s$
if $f = d(s)$ and
$(b_s u_s) c_f = 0$ otherwise
for all $b_s \in A_{c(s)}$, $c_f \in A_f$,
$f \in ob(G)$, $s \in G$.
By (\ref{algebraa}) this $A$-bimodule structure
makes $A \rtimes_{\alpha}^{\sigma} G$ an
$A$-algebra. We will often identify $A$ with
$\oplus_{e \in ob(G)} A_e u_e$; this ring will be referred
to as the coefficient ring of $A \rtimes_{\alpha}^{\sigma} G$.
It is clear that $A \rtimes_{\alpha}^{\sigma} G$ is a category graded ring
in the sense defined in \cite{lu06} and it is strongly graded
if and only if each $\alpha(s,t)$, $(s,t) \in G^{(2)}$,
has a left inverse in $A_{c(s)}$.
We call $A \rtimes_{\alpha}^{\sigma} G$ the category
crossed algebra associated to the crossed system
$\{ A,G,\sigma,\alpha \}$.

In Section 2, we determine the center of category crossed
products. In particular, we determine when category
crossed products are commutative.
In Section 3, we describe the commutant of the
coefficient ring in category crossed products.
In Section 4, we investigate the connection between on the one hand
maximal commutativity of the coefficient ring
and on the other hand
nonemptyness of intersections of the coefficient ring
by nonzero twosided ideals.
In the end of each section, we indicate how our results
generalize earlier results
for other algebraic structures such as group crossed products and matrix rings
(see Remarks 1-6 and Remark 8).

\section{The Center}

For the rest of the article, unless otherwise stated,
we suppose that $A \rtimes_{\alpha}^{\sigma} G$ is a category
crossed product.
We say that $\alpha$ is symmetric if
$\alpha(s,t) = \alpha(t,s)$ for all $s,t \in G$
with $d(s)=c(s)=d(t)=c(t)$.
We say that $A \rtimes_{\alpha}^{\sigma} G$
is a monoid (groupoid, group) crossed product if $G$
is a monoid (groupoid, group).
We say that $A \rtimes_{\alpha}^{\sigma} G$ is a
twisted category (monoid, groupoid, group) algebra
if each $\sigma_s$, $s \in G$, with $d(s)=c(s)$
equals the identity map on $A_{d(s)}=A_{c(s)}$;
in that case the category (monoid, groupoid, group) crossed product is
denoted $A \rtimes_{\alpha} G$.
We say that $A \rtimes_{\alpha}^{\sigma} G$
is a skew category (monoid, groupoid, group) algebra
if $\alpha(s,t) = 1_{A_{c(s)}}$, $(s,t) \in G^{(2)}$;
in that case the category (monoid, groupoid, group) crossed product is
denoted $A \rtimes^{\sigma} G$.
If $G$ is a monoid, then we let $A^G$ denote
the set of elements in $A$ fixed by all $\sigma_s$, $s \in G$.
We say that $G$ is cancellable if any equality
of the form $s_1 t_1 = s_2 t_2$,
$(s_i,t_i) \in G^{(2)}$, $i=1,2$,
implies that $s_1 = s_2$ (or $t_1 = t_2$)
whenever $t_1 = t_2$ (or $s_1 = s_2$).
For $e,f \in ob(G)$ we let $G_{f,e}$ denote the
collection of $s \in G$ with $c(s) = f$ and
$d(s) = e$; we let $G_e$ denote the monoid
$G_{e,e}$. We let the restriction of
$\alpha$ (or $\sigma$) to $G_e^2$
(or $G_e$) be denoted by $\alpha_e$
(or $\sigma_e$). With this notation all
$A_e \rtimes_{\alpha_e}^{\sigma_e} G_e$, $e \in ob(G)$,
are monoid crossed products.

\begin{prop}\label{centermonoid}
The center of a monoid crossed product
$A \rtimes_{\alpha}^{\sigma} G$ is the collection of
$\sum_{s \in G} a_s u_s$ in $A \rtimes_{\alpha}^{\sigma} G$
satisfying the following two conditions:
(i) $a_s \sigma_s(a) = a a_s$, $s \in G$, $a \in A$;
(ii) for all $t,r \in G$ the following equality holds
$\sum_{s \in G \atop st=r} a_s \alpha(s,t) =
\sum_{s \in G \atop ts=r} \sigma_t(a_s) \alpha(t,s)$.
\end{prop}

\begin{proof}
Let $e$ denote the identity element of $G$.
Take $x := \sum_{s \in G} a_s u_s$
in the center of $A \rtimes_{\alpha}^{\sigma} G$.
Condition (i) follows from the fact that
$x au_e = au_e x$ for all $a \in A$.
Condition (ii) follows from the fact that
$x u_t = u_t x$ for all $t \in G$. Conversely, it is clear
that conditions (i) and (ii) are sufficient for
$x$ to be in the center of $A \rtimes_{\alpha}^{\sigma} G$.
\end{proof}

\begin{cor}\label{centertwisted}
The center of a twisted monoid ring
$A \rtimes_{\alpha} G$ is the collection of
$\sum_{s \in G} a_s u_s$ in $A \rtimes_{\alpha} G$
satisfying the following two conditions:
(i) $a_s \in Z(A)$, $s \in G$;
(ii) for all $t,r \in G$, the following equality holds
$\sum_{s \in G \atop st=r} a_s \alpha(s,t) =
\sum_{s \in G \atop ts=r} a_s \alpha(t,s)$.
\end{cor}

\begin{proof}
This follows immediately from
Proposition \ref{centermonoid}.
\end{proof}

\begin{cor}\label{centerskew}
If $G$ is an abelian cancellable monoid,
$\alpha$ is symmetric and has the
property that none of the $\alpha(s,t)$, $(s,t) \in G^{(2)}$,
is a zerodivisor, then
the center of $A \rtimes_{\alpha}^{\sigma} G$ is the collection of
$\sum_{s \in G} a_s u_s$ in $A \rtimes_{\alpha}^{\sigma} G$
satisfying the following two conditions:
(i) $a_s \sigma_s(a) = a a_s$, $s \in G$, $a \in A$;
(ii) $a_s \in A^G$, $s \in G$.
In particular, if $A \rtimes^{\sigma} G$ is a
skew monoid ring where $G$ is abelian and cancellable,
then the same description of the center is valid.
\end{cor}

\begin{proof}
Take $x := \sum_{s \in G} a_s u_s$ in
$A \rtimes_{\alpha}^{\sigma} G$.
Suppose that $x$ belongs to the center of $A \rtimes_{\alpha}^{\sigma} G$.
Condition (i) follows from the first
part of Proposition \ref{centermonoid}.
Now we show condition (ii).
Take $s,t \in G$ and let $r:=st$.
Since $G$ is commutative and cancellable, we get,
by the second part of Proposition \ref{centermonoid}, that
$a_s \alpha(s,t) = \sigma_t(a_s) \alpha(t,s)$.
Since $\alpha$ is symmetric and $\alpha(s,t)$
is not a zerodivisor, this implies that
$a_s = \sigma_t(a_s)$.
Since $s$ and $t$ were arbitrarily chosen from $G$, this
implies that $a_s \in A^G$, $s \in G$.
On the other hand, by Proposition \ref{centermonoid},
it is clear that (i) and (ii) are sufficient
conditions for $x$ to be in the center of $A \rtimes_{\alpha}^{\sigma} G$.
The second part of the claim is obvious.
\end{proof}

Now we show that the center of a category crossed
product is a particular subring of the direct sum of the
centers of the corresponding monoid crossed products.

\begin{prop}\label{centergeneral}
The center of a category crossed product
$A \rtimes_{\alpha}^{\sigma} G$ equals the collection of
$\sum_{e \in ob(G)} \sum_{s \in G_e} a_s u_s$ in
$\sum_{e \in ob(G)} Z(A_e \rtimes_{\alpha_e}^{\sigma_e} G_e)$
satisfying
$\sum_{s \in G_e \atop rs=g} \sigma_r(a_s) \alpha(r,s) =
\sum_{t \in G_f \atop tr=g} a_t \alpha(t,r)$
for all $e,f \in ob(G)$, $e \neq f$, and
all $r,g \in G_{f,e}$.
\end{prop}

\begin{proof}
Take $x := \sum_{s \in G} a_s u_s$ in the center of
$A \rtimes_{\alpha}^{\sigma} G$. By the equalities
$u_e x = x u_e$, $e \in ob(G)$, it follows that $a_s = 0$
for all $s \in G$ with $d(s) \neq c(s)$.
Therefore we can write
$x = \sum_{e \in ob(G)} \sum_{s \in G_e} a_s u_s$
where $\sum_{s \in G_e} a_s u_s \in
Z(A_e \rtimes_{\alpha_e}^{\sigma_e} G_e)$, $e \in ob(G)$.
The last part of the claim follows from the fact that the equality
$u_r \left( \sum_{s \in G_e} a_s u_s \right) =
\left( \sum_{s \in G_e} a_s u_s \right) u_r$
holds for all $e,f \in ob(G)$, $e \neq f$, and all
$r \in G_{f,e}$.
\end{proof}

\begin{prop}\label{commutative}
Suppose that $A \rtimes_{\alpha}^{\sigma} G$ is a category
crossed product and consider the following five conditions:
(0) all $\alpha(s,t)$, $(s,t) \in G^{(2)}$, are nonzero;
(i) $A \rtimes_{\alpha}^{\sigma} G$ is commutative;
(ii) $G$ is the disjoint union of the monoids $G_e$, $e \in ob(G)$,
and they are all abelian;
(iii) each $A_e \rtimes_{\alpha_e}^{\sigma_e} G_e$, $e \in ob(G)$,
is a twisted monoid algebra;
(iv) $A$ is commutative;
(v) $\alpha$ is symmetric. Then
(a) Conditions (0) and (i) imply conditions (ii)-(v);
(b) Conditions (ii)-(v) imply condition (i).
\end{prop}

\begin{proof}
(a) Suppose that conditions (0) and (i) hold.
By Proposition \ref{centergeneral}, we get that
$G$ is the direct sum of $G_e$, $e \in ob(G)$,
and that each $A_e \rtimes_{\alpha_e}^{\sigma_e} G_e$, $e \in ob(G)$,
is commutative. The latter and Proposition \ref{centermonoid}(i)
imply that (iii) holds.
Corollary \ref{centertwisted} now implies that (iv) holds.
For the rest of the proof we can suppose that $G$ is a monoid.
Take $s,t \in G$. By the commutativity of
$A \rtimes_{\alpha}^{\sigma} G$ we get that
$\alpha(s,t) u_{st} = u_s u_t = u_t u_s = \alpha(t,s) u_{ts}$
for all $s,t \in G$.
Since $\alpha$ is nonzero this implies that $st=ts$ and that
$\alpha(s,t)=\alpha(t,s)$ for all $s,t \in G$.
Therefore, $G$ is abelian and (v) holds.

Conversely, by Corollary \ref{centertwisted} and
Corollary \ref{centerskew} we get that conditions (ii)-(iv)
are sufficient for commutativity of
$A \rtimes_{\alpha}^{\sigma} G$.
\end{proof}

\begin{rem}\label{firstremark}
Proposition \ref{centergeneral}, Corollary \ref{centertwisted},
Corollary \ref{centerskew} and Proposition \ref{commutative}
generalize Proposition 3 and
Corollaries 1-4 in \cite{oin06} from groups to categories.
\end{rem}

\begin{rem}\label{secondremark}
Let $A \rtimes G$ be a category algebra
where all the rings $A_e$, $e \in ob(G)$, coincide with
a fixed ring $D$.
Then $A \rtimes G$ is the usual category algebra
$DG$ of $G$ over $D$. Let $H$ denote the disjoint
union of the monoids $G_e$, $e \in ob(G)$.
By Proposition \ref{centermonoid} and Proposition \ref{centergeneral}
the center of $DG$ is
the collection of $\sum_{s \in H} a_s u_s$, $a_s \in Z(D)$,
$s \in H$, in the induced category algebra $Z(D)H$ satisfying
$\sum_{s \in H \atop st=r} a_s =
\sum_{s \in H \atop ts=r} a_s$ for all $r,t \in G$.
Note that if $G$ is a groupoid, then the last condition
simplifies to $a_{rt^{-1}} = a_{t^{-1}r}$
for all $r,t \in G$ with $c(r)=c(t)$ and $d(r)=d(t)$.
This result specializes to two well known cases.
First of all, if $G$ is a group, then we retrieve
the usual description of the center of a group ring
(see e.g. \cite{pas77}). Secondly, if
$G$ is the groupoid with the $n$ first positive integers as objects and
as arrows all pairs $(i,j)$, $1 \leq i,j \leq n$,
equipped with the partial binary operation
defined by letting $(i,j)(k,l)$ be defined and equal to $(i,l)$
precisely when $j=k$, then $DG$ is the ring of square matrices
over $D$ of size $n$ and we retrieve the result that $Z(M_n(D))$
equals the $Z(D)1_n$ where $1_n$ is the unit $n \times n$ matrix.
\end{rem}

\begin{rem}\label{thirdremark}
Let $L/K$ be a finite separable (not necessarily normal) field
extension. Let $N$ denote a normal closure of $L/K$ and
let $Gal$ denote the Galois group of $N/K$.
Furthermore, let $F$ denote the direct sum of the conjugate fields
$L_i$, $i = 1, \ldots , n$; put $L_1 = L$.
If $1 \leq i,j \leq n$, then let $G_{ij}$ denote the set of
field isomorphisms from $L_j$ to $L_i$. If $s \in G_{ij}$,
then we indicate this by writing $d(s) = j$ and $c(s) = i$.
If we let $G$ be the union of the $G_{ij}$,
$1 \leq i,j \leq n$, then $G$ is a groupoid.
For each $s \in G$, let $\sigma_s = s$.
Suppose that $\alpha$ is a map
$G^{(2)} \rightarrow \bigsqcup_{i=1}^n L_i$
with $\alpha(s,t) \in L_{c(s)}$,
$(s,t) \in G^{(2)}$ satisfying
(\ref{identity}), (\ref{associative}) and (\ref{algebra})
for all $(s,t,r) \in G^{(3)}$ and all $a \in L_{d(t)}$.
The category crossed product
$F \rtimes_{\alpha}^{\sigma} G$
extends the construction usually defined by
Galois field extensions $L/K$.
By Proposition \ref{centergeneral},
the center of $F \rtimes_{\alpha}^{\sigma} G$ is
the collection of $\sum_{e \in ob(G)} a_e u_e$
with $a_e = s(a_f)$ for all $e,f \in ob(G)$
and all $s \in G$ with $c(s)=e$ and $d(s)=f$.
Therefore the center is a field isomorphic
to $L^{G_{1,1}}$ and we retrieve the first part
of Theorem 4 in \cite{lu05}.
\end{rem}

\section{The commutant of the coefficient ring}

\begin{prop}\label{commutant}
The commutant of $A$ in $A \rtimes_{\alpha}^{\sigma} G$
is the collection of $\sum_{s \in G} a_s u_s$ in
$A \rtimes_{\alpha}^{\sigma} G$ satisfying
$a_s = 0$, $s \in G$, $d(s) \neq c(s)$, and
$a_s \sigma_s(a) = a a_s$, $s \in G$, $d(s)=c(s)$,
$a \in A_{d(s)}$.
\end{prop}

\begin{proof}
The first claim follows from the fact that the equality
$(\sum_{s \in G} a_s u_s)u_e = u_e(\sum_{s \in G} a_s u_s)$
holds for all $e \in ob(G)$.
The second claim follows from the fact that the equality
$(\sum_{s \in G} a_s u_s)au_e = au_e(\sum_{s \in G} a_s u_s)$
holds for all $e \in ob(G)$ and all $a \in A_e$.
\end{proof}

Recall that the annihilator
of an element $r$ in a commutative ring $R$
is the collection, denoted ${\rm ann}(r)$, of elements
$s$ in $R$ with the property that $rs = 0$.

\begin{cor}\label{commutativeA}
Suppose that $A$ is commutative. Then
the commutant of $A$ in $A \rtimes_{\alpha}^{\sigma} G$
is the collection of $\sum_{s \in G} a_s u_s$ in
$A \rtimes_{\alpha}^{\sigma} G$ satisfying
$a_s = 0$, $s \in G$, $d(s) \neq c(s)$, and
$\sigma_s(a)-a \in {\rm ann}(a_s)$, $s \in G$,
$d(s)=c(s)$, $a \in A_{d(s)}$.
In particular, $A$ is maximal commutative in
$A \rtimes_{\alpha}^{\sigma} G$ if and only if
there for all choices of $e \in ob(G)$,
$s \in G_e \setminus \{ e \}$, $a_s \in A_e$,
there is a nonzero $a \in A_e$ with the
property that $\sigma_s(a)-a \notin {\rm ann}(a_s)$.
\end{cor}

\begin{proof}
This follows immediately from Proposition \ref{commutant}.
\end{proof}

\begin{cor}\label{integraldomain}
Suppose that each $A_e$, $e \in ob(G)$, is an integral domain.
Then the commutant of $A$ in $A \rtimes_{\alpha}^{\sigma} G$
is the collection of $\sum_{s \in G} a_s u_s$ in
$A \rtimes_{\alpha}^{\sigma} G$ satisfying
$a_s = 0$ whenever $\sigma_s$ is not an identity map.
In particular, $A$ is maximal commutative in
$A \rtimes_{\alpha}^{\sigma} G$ if and only if
for all nonidentity $s \in G$, the map $\sigma_s$
is not an identity map.
\end{cor}

\begin{proof}
This follows immediately from Corollary \ref{commutativeA}.
\end{proof}

\begin{prop}\label{alphasymmetric}
If $A$ is commutative,
$G$ a disjoint union of abelian monoids and
$\alpha$ is symmetric, then the commutant of $A$
in $A \rtimes_{\alpha}^{\sigma} G$ is the unique
maximal commutative subalgebra of $A \rtimes_{\alpha}^{\sigma} G$
containing $A$.
\end{prop}

\begin{proof}
We need to show that the commutant of $A$
in $A \rtimes_{\alpha}^{\sigma} G$ is commutative.
By the first part of Proposition \ref{commutant},
we can assume that $G$ is an abelian monoid.
If we take $\sum_{s \in G} a_s u_s$ and
$\sum_{t \in G} b_t u_t$ in the commutant of
$A$ in $A \rtimes_{\alpha}^{\sigma} G$, then,
by the second part of Proposition \ref{commutant}
and the fact that $\alpha$ is symmetric,
we get that
$$\sum_{s \in G} a_s u_s \sum_{t \in G} b_t u_t
= \sum_{s,t \in G} a_s \sigma_s(b_t) \alpha(s,t) u_{st}
= \sum_{s,t \in G} a_s b_t \alpha(s,t) u_{st} =$$
$$= \sum_{s,t \in G} b_t a_s \alpha(t,s) u_{ts}
= \sum_{s,t \in G} b_t \sigma_t(a_s) \alpha(t,s) u_{st}
= \sum_{t \in G} b_t u_t   \sum_{s \in G} a_s u_s$$
\end{proof}

\begin{rem}\label{fourthremark}
Proposition \ref{commutant}, Corollary \ref{commutativeA},
Corollary \ref{integraldomain} and Proposition \ref{alphasymmetric}
together generalize Theorem 1, Corollaries 5-10 and Proposition 4
in \cite{oin06} from groups to categories.
\end{rem}

\begin{rem}
Let $A \rtimes G$ be a category algebra
where all the rings $A_e$, $e \in ob(G)$, coincide with
a fixed integral domain $D$.
Then $A \rtimes G$ is the usual category algebra
$DG$ of $G$ over $D$.
By Corollary \ref{integraldomain}, the commutant
of $D$ in $DG$ is $DG$ itself. In particular, $A$ is
maximal commutative in $DG$ if and only if
$G$ is the disjoint union of $|ob(G)|$ copies
of the trivial group.
\end{rem}

\begin{rem}
Let $L/K$ be a finite separable (not necessarily normal) field
extension. We use the same notation as in Remark \ref{thirdremark}.
By Corollary \ref{integraldomain}, the commutant of $F$
in $F \rtimes_{\alpha}^{\sigma} G$ is the collection
of $\sum_{i = 1}^n \sum_{s \in G_{ii}} a_s u_s$
satisfying $a_s = 0$ whenever $\sigma_s$ is not an identity map.
In particular, $F$ is maximal commutative in $F \rtimes_{\alpha}^{\sigma} G$
if all groups $G_{i,i}$, $i = 1, \ldots , n$, are nontrivial;
this of course happens in the case when $L/K$ is a Galois
field extension.
\end{rem}

\section{Commutativity and Ideals}

In this section, we investigate the connection between on the one hand
maximal commutativity of the coefficient ring
and on the other hand
non\-empty\-ness of intersections of the coefficient ring
by nonzero twosided ideals.
For the rest of the article, we assume that $ob(G)$
is finite. Recall (from Section 1) that this is equivalent to the fact
that $A \rtimes_{\alpha}^{\sigma} G$ has a multiplicative
identity; in that case the multiplicative identity is
$\sum_{e \in ob(G)} u_e$.

\begin{thm}\label{intersection}
If $A \rtimes_{\alpha}^{\sigma} G$ is a groupoid crossed product
with $A$ commutative and for every $s \in G$,
$\alpha(s,s^{-1})$ is not a zero divisor in $A_{c(s)}$, then every
intersection of a nonzero twosided ideal of $A \rtimes_{\alpha}^{\sigma} G$
with the commutant of $A$ in $A \rtimes_{\alpha}^{\sigma} G$ is
nonzero.
\end{thm}

\begin{proof}
We show the contrapositive statement. Let $C$ denote the commutant
of $A$ in $A \rtimes_{\alpha}^{\sigma} G$ and suppose that $I$ is a
twosided ideal of $A \rtimes_{\alpha}^{\sigma} G$ with the property that $I
\cap C = \{ 0 \}$. We wish to show that $I = \{ 0 \}$. Take $x \in
I$. If $x \in C$, then by the assumption $x = 0$. Therefore we now
assume that $x  = \sum_{s \in G} a_s u_s  \in I$, $a_s \in
A_{c(s)}$, $s \in G$, and that $x$ is chosen so that $x \notin C$
with the set $S := \{ s \in G \mid a_s \neq 0 \}$ of least possible
cardinality $N$. Seeking a contradiction, suppose that $N$ is
positive. First note that there is $e \in ob(G)$ with $u_e x \in I
\setminus C$. In fact, if $u_e x \in C$ for all $e \in ob(G)$,
then $x = 1x = \sum_{e \in ob(G)} u_e x \in C$ which is a
contradiction. By minimality of $N$ we can assume that $c(s)=e$, $s
\in S$, for some fixed $e \in ob(G)$. Take $t \in S$ and consider
the element $x' := x u_{t^{-1}} \in I$. Since $\alpha(t,t^{-1})$
is not a zero divisor we get that $x' \neq 0$. Therefore, since
$I \cap C = \{ 0 \}$, we get that $x' \in I \setminus C$. Take $a =
\sum_{f \in ob(G)} b_f u_f \in A$. Then $I \ni x'' := ax' - x'a =
\sum_{s \in S} a_s(b_{d(s)} - \sigma_s(b_e)) u_s $. Since the
summand for $s = e$ vanishes, we get, by the assumption on $N$, that
$x'' = 0$. Since $a \in A$ was arbitrarily chosen, we get that $x'
\in C$ which is a contradiction. Therefore $N = 0$ and hence $S =
\emptyset$ which in turn implies that $x=0$. Since $x \in I$ was
arbitrarily chosen, we finally get that $I = \{ 0 \}$.
\end{proof}

\begin{cor}\label{intersectioncorollary}
If $A \rtimes_{\alpha}^{\sigma} G$ is a groupoid crossed product
with $A$ maximal commutative and for every $s \in G$,
$\alpha(s,s^{-1})$ is not a zero divisor in $A_{c(s)}$, then every
intersection of a nonzero twosided ideal of $A \rtimes_{\alpha}^{\sigma} G$
with $A$ is nonzero.
\end{cor}

\begin{proof}
This follows immediately from Theorem \ref{intersection}.
\end{proof}

Now we examine conditions under which the opposite statement of
Corollary \ref{intersectioncorollary} is true. To this end, we recall some
notions from category theory that we need in the sequel (for the
details see e.g. \cite{mac}). Let $G$ be a category. A congruence
relation $R$ on $G$ is a collection of equivalence relations
$R_{a,b}$ on $hom(a,b)$, $a,b \in ob(G)$, chosen so that if $(s,s')
\in R_{a,b}$ and $(t,t') \in R_{b,c}$, then $(ts,t's') \in R_{a,c}$
for all $a,b,c \in ob(G)$. Given a congruence relation $R$ on $G$ we
can define the corresponding quotient category $G/R$ as the category
having as objects the objects of $G$ and as arrows the corresponding
equivalence classes of arrows from $G$. In that case there is a
corresponding functor $\mathcal{Q}_R : G \rightarrow G/R$ mapping objects with
the identity functor and mapping arrows to their respective
equivalence classes. We will often use the notation $[s] := \mathcal{Q}_R(s)$,
$s \in G$. Suppose that $H$ is another category and that $F : G
\rightarrow H$ is a functor. The kernel of $F$, denoted $ker(F)$, is
the congruence relation on $G$ defined by letting $(s,t) \in
ker(F)_{a,b}$, $a,b \in ob(G)$, whenever $s,t \in hom(a,b)$ and
$F(s)=F(t)$. In that case there is a unique functor $\mathcal{P}_F : G/ker(F)
\rightarrow H$ with the property that $\mathcal{P}_F \mathcal{Q}_{ker(F)} = F$.
Furthermore, if there is a congruence relation $R$ on $G$ contained
in $ker(F)$, then there is a unique functor $\mathcal{N} : G/R \rightarrow
G/ker(F)$ with the property that $\mathcal{N} \mathcal{Q}_R = \mathcal{Q}_{ker(F)}$. In that case
there is therefore always a factorization $F = \mathcal{P}_F \mathcal{N} \mathcal{Q}_R$; we will
refer to this factorization as the canonical one.

\begin{prop}\label{induced}
Let $\{ A,G,\sigma,\alpha \}$ and
$\{ A,H,\tau,\beta \}$ be crossed systems with $ob(G)=ob(H)$.
Suppose that there is a functor $F : G \rightarrow H$
satisfying the following three criteria:
(i) $F$ is the identity map on objects;
(ii) $\tau_{F(s)} = \sigma_s$, $s \in G$;
(iii) $\beta(F(s),F(t)) = \alpha(s,t)$, $(s,t) \in G^{(2)}$.
Then there is a unique $A$-algebra homomorphism
$A \rtimes_{\alpha}^{\sigma} G
\rightarrow A \rtimes_{\beta}^{\tau} H$, also denoted $F$,
satisfying $F(u_s) = u_{F(s)}$, $s \in G$.
\end{prop}

\begin{proof}
Take $x := \sum_{s \in G} a_s u_s$ in $A \rtimes_{\alpha}^{\sigma} G$
where $a_s \in A_{c(s)}$, $s \in G$.
By $A$-linearity we get that
$F(x) = \sum_{s \in G} a_s F(u_s)=
\sum_{s \in G} a_s u_{F(s)}$.
Therefore $F$ is unique.
It is clear that $F$ is additive.
By (i), $F$ respects the multiplicative identities.
Now we show that $F$ is multiplicative.
Take another $y := \sum_{s \in G} b_s u_s$ in
$A \rtimes_{\alpha}^{\sigma} G$ where $b_s \in A_{c(s)}$, $s \in G$.
Then, by (ii) and (iii), we get that
$$F(xy) =
F \left( \sum_{(s,t) \in G^{(2)}}  a_s \sigma_s(b_t) \alpha(s,t) u_{st} \right)
= \sum_{(s,t) \in G^{(2)}} a_s \sigma_s(b_t) \alpha(s,t) u_{F(st)} =$$
$$= \sum_{(s,t) \in G^{(2)}} a_s \tau_{F(s)}(b_t) \beta(F(s),F(t)) u_{F(s)F(t)}
= F(x) F(y)$$
\end{proof}

\begin{rem}\label{functor}
Suppose that $\{ A,G,\sigma,\alpha \}$ is a crossed system.
By abuse of notation, we let $A$ denote the category
with the rings $A_e$, $e \in ob(G)$, as objects and
ring homomorphisms $A_e \rightarrow A_f$, $e,f \in ob(G)$,
as morphisms. Define a map $\sigma : G \rightarrow A$
on objects by $\sigma(e) = A_e$, $e \in ob(G)$, and on
arrows by $\sigma(s) = \sigma_s$, $s \in G$.
By equation (\ref{algebra}) it is clear that
$\sigma$ is a functor if the following two conditions
are satisfied:
(i) for all $(s,t) \in G^{(2)}$, $\alpha(s,t)$
belongs to the center of $A_{c(s)}$;
(ii) for all $(s,t) \in G^{(2)}$, $\alpha(s,t)$
is not a zero divisor in $A_{c(s)}$.
\end{rem}

\begin{prop}\label{mainprop}
Let $A \rtimes_{\alpha}^{\sigma} G$ be a category crossed product
with $\sigma : G \rightarrow A$ a functor. Suppose that $R$ is a
congruence relation on $G$ with the property that the associated
quadruple $\{ A,G/R,\sigma([\cdot]),\alpha([\cdot],[\cdot]) \}$ is a
crossed system. If $I$ is the twosided ideal in $A \rtimes_{\alpha}^{\sigma}
G$ generated by an element $\sum_{s \in G} a_s u_s$, $a_s \in
A_{c(s)}$, $s \in G$, satisfying $a_s = 0$ if $s$ does not belong to
any of the classes $[e]$, $e \in ob(G)$, and $\sum_{s \in [e]} a_s =
0$, $e \in ob(G)$, then $A \cap I = \{ 0 \}$.
\end{prop}

\begin{proof}
By Proposition \ref{induced}, the functor $\mathcal{Q}_R$ induces an
$A$-algebra homomorphism $\mathcal{Q}_R : A \rtimes_{\alpha}^{\sigma} G
\rightarrow A \rtimes_{\alpha([\cdot],[\cdot])}^{\sigma([\cdot])}
G/R$. By definition of the $a_s$, $s \in G$, we get that
$$\mathcal{Q}_R \left( \sum_{s \in G} a_s u_s \right)
= \mathcal{Q}_R \left( \sum_{e \in ob(G)} \sum_{s \in [e]} a_s u_s \right)=$$
$$= \sum_{e \in ob(G)} \sum_{s \in [e]} a_s u_{[s]} =
\sum_{e \in ob(G)} \left( \sum_{s \in [e]} a_s \right) u_{[e]} = 0$$
This implies that $\mathcal{Q}_R(I) = \{ 0 \}$.
Since $\mathcal{Q}_R|_{A} = id_A$, we therefore get that
$I \cap A = (\mathcal{Q}_R|_A) (A \cap I) \subseteq \mathcal{Q}_R(I) = \{ 0 \}$.
\end{proof}

Let $G$ be a groupoid and suppose that we for each $e \in ob(G)$
are given a subgroup $N_e$ of $G_e$. We say that
$N = \cup_{e \in ob(G)} N_e$ is a normal subgroupoid of $G$
if $sN_{d(s)} = N_{c(s)}s$ for all $s \in G$. The normal
subgroupoid $N$ induces a congruence relation $\sim$ on $G$ defined
by letting $s \sim t$, $s,t \in G$, if there is $n$ in
$N_{d(t)}$ with $s = n t$. The corresponding quotient category is a
groupoid which is denoted $G/N$. For more details, see e.g.
\cite{hig}; note that our definition of normal subgroupoids
is more restrictive than the one used in \cite{hig}.

\begin{prop}\label{normal}
Let $A \rtimes_{\alpha}^{\sigma} G$ be a groupoid crossed product
such that for each $(s,t) \in G^{(2)}$,
$\alpha(s,t) \in Z(A_{c(s)})$ and
$\alpha(s,t)$ is not a zero divisor in $A_{c(s)}$.
Suppose that $N$ is a normal
subgroupoid of $G$ with the property that
$\sigma_n = id_{A_{c(n)}}$, $n \in
N$, and $\alpha(s,t)=1_{A_{c(s)}}$ if $s \in N$ or $t \in N$. If $I$ is the
twosided ideal in $A \rtimes_{\alpha}^{\sigma} G$ generated by an element
$\sum_{s \in G} a_s u_s$, $a_s \in A_{c(s)}$, $s \in G$, satisfying
$a_s = 0$ if $s$ does not belong to any of the sets $N_e$, $e \in
ob(G)$, and $\sum_{s \in N_e} a_s = 0$, $e \in ob(G)$, then $A \cap
I = \{ 0 \}$.
\end{prop}

\begin{proof}
By Remark \ref{functor}, $\sigma$ is a functor $G \rightarrow A$ and
$\sim \ \subseteq ker(\sigma)$. Therefore, by the discussion
preceding Proposition \ref{induced}, there is a well defined
functor $\sigma[\cdot] : G/N \rightarrow A$. Now we show that the
induced map $\alpha([\cdot],[\cdot])$ is well defined. By equation
(\ref{associative}) with $s = n \in N_{c(t)}$ we get that $\alpha(n,t)
\alpha(nt,r) = \sigma_n(\alpha(t,r)) \alpha(n,tr)$. By the
assumptions on $\alpha$ and $\sigma$ we get that
$\alpha(nt,r)=\alpha(t,r)$.
Analogously, by equation (\ref{associative}) with $t = n \in N_{d(r)}$,
we get that $\alpha(s,t) = \alpha(s,tn)$.
Therefore, $\alpha([\cdot],[\cdot])$ is well defined.
The rest of the claim now follows
immediately from Proposition \ref{mainprop}.
\end{proof}

\begin{prop}\label{congruence}
Let $A \rtimes^{\sigma} G$ be a skew category algebra.
Suppose that $R$ is a congruence relation on $G$ contained
in $ker(\sigma)$.
If $I$ is the twosided ideal in $A \rtimes^{\sigma} G$
generated by an element $\sum_{s \in G} a_s u_s$,
$a_s \in A_{c(s)}$, $s \in G$, satisfying
$a_s = 0$ if $s$ does not belong to any of the classes
$[e]$, $e \in ob(G)$, and
$\sum_{s \in [e]} a_s = 0$, $e \in ob(G)$,
then $A \cap I = \{ 0 \}$.
\end{prop}

\begin{proof}
By Remark \ref{functor} and the discussion preceding Proposition
\ref{induced}, there is a well defined functor $\sigma[\cdot] : G/R
\rightarrow A$. The claim now follows immediately from Proposition
\ref{mainprop}.
\end{proof}

\begin{prop}\label{skewgroupoid}
Let $A \rtimes^{\sigma} G$ be a skew groupoid ring with all $A_e$,
$e \in ob(G)$, equal integral domains and each $G_e$, $e \in ob(G)$,
an abelian group. If every intersection of a nonzero twosided ideal of $A
\rtimes^{\sigma} G$ and $A$ is nonzero, then $A$ is maximal
commutative in $A \rtimes^{\sigma} G$.
\end{prop}

\begin{proof}
We show the contrapositive statement. Suppose that $A$ is not
maximal commutative in $A \rtimes^{\sigma} G$. By the second part of
Corollary \ref{integraldomain}, there is $e \in ob(G)$ and a nonidentity
$s \in G_e$ such that $\sigma_s = id_{A_e}$. Let $N_e$ denote the
cyclic subgroup of $G_e$ generated by $s$. Note that since $G_e$ is
abelian, $N_e$ is a normal subgroup of $G_e$. For each $f \in
ob(G)$, define a subgroup $N_f$ of $G_f$ in the following way. If
$G_{e,f} \neq \emptyset$, then let $N_f = s N_e s^{-1}$, where $s$
is a morphism in $G_{e,f}$. If, on the other hand, $G_{e,f} =
\emptyset$, then let $N_f = \{ f \}$. Note that if $s_1 , s_2 \in
G_{e,f}$, then $s_2^{-1} s_1 \in G_e$ and hence $s_1 N_e s_1^{-1} =
s_2 s_2^{-1} s_1 N_e (s_2^{-1}s_1)^{-1} s_2^{-1} = s_2 N_e
s_2^{-1}$. Therefore, $N_f$ is well defined. Now put $N = \cup_{f
\in ob(G)} N_f$. It is clear that $N$ is a normal subgroupoid of $G$
and that $\sigma_n = id_{A_e}$, $n \in N$. Let $I$ be the nonzero
twosided ideal of $A \rtimes^{\sigma} G$ generated by $u_e - u_s$. By
Proposition \ref{normal} (or Proposition \ref{congruence}) it
follows that $A \cap I = \{ 0 \}$.
\end{proof}

\begin{rem}\label{fifthremark}
Proposition \ref{intersection}, Corollary \ref{intersectioncorollary}
and Propositions \ref{mainprop}-\ref{skewgroupoid}
together generalize Theorem 2, Corollary 11,
Theorem 3, Corollaries 12-15 and Theorem 4 in \cite{oin06} from groups to categories.
\end{rem}

By combining Theorem 2 and Proposition \ref{skewgroupoid},
we get the follwing result.

\begin{cor}
If $A \rtimes^{\sigma} G$ is a skew groupoid ring with all $A_e$,
$e \in ob(G)$, equal integral domains and each $G_e$, $e \in ob(G)$,
an abelian group, then $A$ is maximal commutative in $A \rtimes^{\sigma} G$
if and only if every intersection of a nonzero twosided ideal of $A
\rtimes^{\sigma} G$ and $A$ is nonzero.
\end{cor}

{\bf Acknowledgements:}
The first author was partially supported by The Swedish Research Council, The Crafoord Foundation, The Royal Physiographic Society in Lund, The Swedish Royal Academy of Sciences, The Swedish Foundation of International Cooperation in Research and Higher Education (STINT) and "LieGrits", a Marie Curie Research Training Network funded by the European Community as project MRTN-CT 2003-505078.

\end{document}